\documentclass[12pt]{article}
\usepackage{graphicx,color,amsmath,amssymb,amsthm,rotating,stmaryrd,datetime,pifont,booktabs}
\usepackage[mathscr]{eucal}
\newtheorem{theorem}{Theorem}

\newtheorem{result}[theorem]{Result}
\newtheorem{lemma}[theorem]{Lemma}

\newcommand{\Label}{\label}
\newcommand{\Labele}{\label}

\newcommand{\B}{\mathcal B}

\newcommand{\PG}{{\textup{PG}}}

\newcommand{\PGammaL}{\textup{P}\Gamma{\textup {L}}}
\newcommand{\Fqq}{\mathbb{F}_{q^2}}

\newcommand{\Fq}{\mathbb{F}_{q}}

\newcommand{\K}{\mathcal{K}}

\newcommand{\C}{\mathcal{C}}

\newcommand{\li}{\ell_\infty}
\newcommand{\U}{\mathcal U}
\newcommand{\pedal}{\textsl{\textup{pedal}}}

\newcommand{\e}{\epsilon}
\newcommand{\Uab}{\mathcal U_{\alpha\beta}}
 \newcommand{\eone}{R_1}
\newcommand{\etwo}{R_2}
\newcommand{\ee}{R}

\begin{document}

\title{The feet of orthogonal Buekenhout-Metz unitals}
\author{S.G. Barwick, W.-A. Jackson and P. Wild}
\date{}

\maketitle

%
%

AMS code: 51E20

Keywords: unital, Buekenhout-Metz unital, feet, pedal set

\begin{abstract} In this article we look at the geometric structure of the feet of an orthogonal Buekenhout-Metz unital $U$ in $\PG(2,q^2)$. We show that the feet of each point form a set of type $(0,1,2,4)$. Further, we discuss the structure of any $4$-secants, and determine exactly when the feet form an arc.
\end{abstract}

\section{Introduction}

A set $\K$   of points in the Desarguesian projective plane $\PG(2,q^2)$ is called a \emph{set of type} $(a_1,\ldots,a_k)$ for positive integers $a_1,\ldots,a_k$ if for every line $\ell$ of $\PG(2,q^2)$, $|\ell\cap\K|\in\{a_1,\ldots,a_k\}$.
If for each $a\in\{a_1,\ldots,a_k\}$, there exists at least one line $\ell$ with $|\ell\cap\K|=a$, then $\K$ is called a 
\emph{set of class} $(a_1,\ldots,a_k)$. A  $k$-\emph{arc} in $\PG(2,q^2)$ is a set of $k$ points of type $(0,1,2)$. 

A \emph{unital} in $\PG(2,q^2)$ is a set of class $(1,q+1)$ of size $q^3+1$.  
The set of absolute points of a unitary polarity forms a unital, called the \emph{classical unital}. 
The known unitals of $\PG(2,q^2)$ can be described using the \emph{Bruck-Bose representation} of $\PG(2,q^2)$ in $\PG(4,q)$. This representation was developed independently by Andr\'e \cite{andre}, Segre~\cite{segre} and Bruck and Bose \cite{bruc64,bruc66}. See \cite{book} for full details on the Bruck-Bose representation, and the structure of unitals in this representation. The Bruck-Bose representation is an affine construction, so the representation depends on the line at infinity. 
Buekenhout~\cite{buek} showed that  the classical unital tangent to $\li$
of $\PG(2,q^2)$ corresponds in $\PG(4,q)$ to  an elliptic cone containing a spread line. Moreover, each   elliptic cone in $\PG(4,q)$  containing a spread line corresponds to a unital of $\PG(4,q)$. Metz~\cite{metz} showed there exists non-classical unitals arising from this construction. To emphasis that the correspondence relates to the line at infinity, we say a unital is an \emph{orthogonal Buekenhout-Metz unital with respect to $\li$} if it corresponds to an elliptic cone in the $\PG(4,q)$ Bruck-Bose representation \emph{with respect to $\li$}. We abbreviate this and write \emph{orthogonal BM unital wrt $\li$}.

More generally, a unital of $\PG(2,q^2)$ that corresponds to an ovoidal cone in the $\PG(4,q)$ Bruck-Bose representation with respect to $\li$  is called an \emph{ovoidal BM  unital wrt $\li$}. In particular, a \emph{Buekenhout-Tits unital} corresponds in $\PG(4,q)$ to an ovoidal cone whose ovoid is the Tits ovoid. 
All known unitals in $\PG(2,q^2)$ are ovoidal BM unitals, see \cite[p66]{book}. 

Coordinates for orthogonal BM unitals in $\PG(2,q^2)$ were computed in \cite{bake92} for $q$ odd and in \cite{eber92} for $q$ even.  Let $\Fq$ denote the finite field with $q$ elements.

\begin{result}\Label{Uab}
\cite{bake92, eber92}
For $\alpha,\beta\in\Fqq$, the point set 
$$ \U_{\alpha\beta}=\{  P_{x,r}=(x,\ \alpha x^2+\beta x^{q+1}+r,\ 1)\,|\, x\in\Fqq,r\in\Fq\}\cup\{  T_\infty=(0,1,0)\}$$
is an orthogonal BM unital wrt $\li$ in $\PG(2,q^2)$ iff the discriminant condition holds, that is:
\begin{itemize}
\item if $q$ is odd, the  \emph{discriminant} $d=(\beta^q-\beta)^2+4\alpha^{q+1}$ is a nonsquare 
in $\Fq$;
\item if $q$ is even, $\beta\notin\Fq$ and the \emph{discriminant} $d=\alpha^{q+1}/(\beta^q+\beta)^2$ has absolute trace $0$. 
\end{itemize}
Conversely, every orthogonal BM unital is equivalent to $\Uab$ for some $\alpha,\beta\in\Fqq$ satisfying the discriminant condition.
Moreover, $\Uab$ is classical iff $\alpha=0$. 
\end{result}

Let $\U$ be a unital of $\PG(2,q^2)$ and $P$ a point not on $\U$. There are $q+1$ lines through $P$ that are tangent to $\U$, the points of contact with $\U$ are called \emph{feet}. The set of $q+1$ points of contact is called the  \emph{pedal set} of $P$ and is denoted $\pedal(P)$.

Let $\U$ be the classical unital. As $\U$ arises from a polarity, $\pedal(P)$ is a collinear set for all points $P\notin\U$.   It was long conjectured that the converse must be true and this conjecture was proved in 1992 by Thas.

\begin{result}\cite{thas}\Label{result-thas} Let $\U$ be a unital in $\PG(2,q^2)$. If for all points $P\notin\U$, the feet of $P$ are collinear, then $\U$ is classical. 
\end{result}  

Aguglia and Ebert~\cite{aguglia} improved this characterisation showing that if the feet of all points $P\notin\U$ on two tangent lines to $\U$ are collinear, then $\U$ is classical. 

Let $\U$ be an ovoidal BM unital wrt $\li$, then \cite{dove96} showed that for each point $P\in\li$, $P\notin\U$, $\pedal(P)$ is a collinear set. 
In \cite[Thm 4.18, Thm 4.29]{book}, the pedal set of orthogonal BM unitals for points not on $\li$ are looked at, and it is shown that they are not collinear sets. A similar result for the Buekenhout-Tits unitals is proved in \cite{eber92}; see also
\cite[Thm 4.33]{book} and \cite[Remark 2]{fvdv}.

\begin{result}\Label{12} Let $\U$ be a non-classical ovoidal BM unital wrt $\li$ in $\PG(2,q^2)$ and let $P$ be any point not in $\U$. Then the feet of $P$ are collinear iff $P\in\li$. 
\end{result}

In the case when $q$ is odd, the structure of $\pedal(P)$ for points $P\notin\li$ was investigated in \cite{apv}.  

\begin{result}\cite{apv}\Label{30} Let $\Uab$ be a non-classical BM unital in $\PG(2,q^2)$, $q$ odd. For each point  $P\notin\Uab\cup\li$: 
\begin{enumerate}
\item $\pedal(P)$ is a set of type $(0,1,2,4)$,
\item $\pedal(P)$ can be partitioned into two arcs,
\item if $\beta\in\Fq$, then $\pedal(P)$ is a set of type $(0,1,2)$. 
\end{enumerate}
\end{result}

In this article we improve this result by determining exactly when $4$-secants exists, answering the first open problem posed in \cite{apv}. We show  that $\pedal(P)$ has a 4-secant iff $\alpha$ is a  nonzero square in $\Fqq$. 
Note that by Lemma~\ref{11}, if $\beta\in\Fq$, then $\alpha$ is a nonsquare in $\Fqq$, although the converse is not true.  That is, there exist unitals with $\alpha$ a nonsquare in $\Fqq$ and $\beta\notin\Fq$. So there are more cases when $\pedal(P)$ is an arc than described in Result~\ref{30}.

The set of $q+1$ lines joining a vertex point $P$ to the points of a Baer subline base is called a \emph{Baer pencil.}
This article proves the following result on the geometric structure of pedal sets when $q$ is odd.
 
\begin{theorem}\Label{19}  Let $\Uab$ be a non-classical BM unital in $\PG(2,q^2)$, $q$ odd. Let $P\notin\Uab\cup\li$, then  the points of $\pedal(P)$ lie on the lines of a  Baer pencil with vertex a point on $\li$. 
\begin{enumerate}
\item If   $\alpha$ is a nonsquare in $\Fqq$, then   $\pedal(P)$ is   an arc.
\item If   $\alpha$ is a nonzero square in $\Fqq$,  then 
\begin{enumerate}
\item    $\pedal(P)$ is a set of class $(0,1,2,4)$,
\item there are at least $\frac{q-3}2$ 4-secants of $\pedal(P)$,
\item  every 4-secant of $\pedal(P)$  contains  the point vertex  of the Baer pencil containing $\pedal(P)$.
\end{enumerate}
\end{enumerate}
\end{theorem}

The case when $q$ is even and the unital is a Buekenhout-Tits unital was looked at in \cite{fvdv}.

\begin{result}\cite{fvdv} Let $\U$ be a Buekenhout-Tits unital in $\PG(2,q^2)$, $q$ even.  Let $P\notin\U\cup\li$, then $\pedal(P)$ is a set of type $(0,1,2,3,4)$. Moreover, for each $i\in\{0,1,2,3,4\}$, there exists a line $\ell$ and a point $P$ with $|\ell\cap\pedal(P)|=i$. 
\end{result}
 
 The remaining case when $\U$ is a  non-classical orthogonal BM unital and $q$ is even is looked at in this article. We show that in this case the pedal set is always an arc. 

\begin{theorem}\Label{24}
Let $\U$ be a non-classical orthogonal BM unital wrt $\li$ in $\PG(2,q^2)$, $q$ even. Let $P\notin\U\cup\li$, then 
\begin{enumerate}
\item  the points of $\pedal(P)$ lie on the lines of a Baer pencil with vertex a point  in $\li$,
\item $\pedal(P)$  is an arc. 
\end{enumerate}
\end{theorem}

To summarise, the main results of this article are Theorems~\ref{19} and Theorem~\ref{24}. This answers an open question   posed in \cite{aguglia} and listed in \cite{book} to determine the geometric structure of pedal sets for   Buekenhout-Metz unitals.  

%

\section{The feet of an orthogonal BM unital for $q$ odd}

 An orthogonal BM unital in $\PG(2,q^2)$  has form $\U_{\alpha\beta}$ defined in Result~\ref{Uab}. 
To determine the structure of the pedal sets of $\U_{\alpha\beta}$ when $q$ is odd, we use the following result.

\begin{result}\Label{21} \cite{bake92}  Let $\Uab$ be an orthogonal BM unital in $\PG(2,q^2)$, $q$ odd.  The subgroup $G$ of $\PGammaL(3,q^2)$   that fixes $T_\infty$ and fixes $\U_{\alpha\beta}$   acts transitively on the points of $\U_{\alpha\beta}\setminus\{T_\infty\}$, acts transitively on the points of $\li\setminus\{T_\infty\}$, and has one or two orbits on the remaining points of $\PG(2,q^2)$. 
 \end{result}
 
 Let $\Uab$ be a non-classical  orthogonal BM unital in $\PG(2,q^2)$, $q$ odd. We work with  the following basis for $\Fqq$. Choose $\{1,\epsilon\}$ as basis for $\Fqq$ over $\Fq$ where $\zeta$ is a primitive element of $\Fqq$, $\epsilon=\zeta^{(q+1)/2}$, so $w=\epsilon^2$ is a primitive element of $\Fq$ and $\epsilon^q=-\epsilon$.

 The points  with coordinates $ \eone=(0,\epsilon,1)$, $ \etwo=(0,w\epsilon,1)$ are not on $\Uab$. Moreover, if there are two orbits on the points of $\PG(2,q^2)\setminus\{\U_{\alpha\beta}\cup\li\}$, then $R_1,R_2$  lie in different orbits. 
In the next two sections, we   determine the structure of $\pedal(\eone)$ and $\pedal(\etwo)$.
 
\subsection{The feet of $ \eone=(0,\epsilon,1)$, $q$ odd}

\begin{lemma}\Label{10}
\begin{enumerate}
\item The feet of $ \eone=(0,\epsilon,1)$ are the points with coordinates given by  $Q_x=(x,2\alpha x^2+(\beta-\beta^q)x^{q+1}+\epsilon ,1)$ where $x\in\Fqq$ satisfies $ \alpha x^2-\alpha^qx^{2q}+(\beta-\beta^q)x^{q+1}+2\epsilon=0.$
\item  The points of $\pedal(\eone)$ lie on the lines of the Baer pencil  joining the point vertex $U_\infty=(1,0,0)$ to the Baer subline $\{   E_{s-\epsilon}=(0,s-\epsilon,1) ,|\,s \in \Fq\cup\{\infty\}\}$.
\end{enumerate}
\end{lemma}

\begin{proof} The coordinates of the tangents to $\U_{\alpha\beta}$ are stated in \cite[Lemma 4.16]{book}. The tangent line $t_{x,r}$ to $\Uab$  at the point with coordinates $  P_{x,r}=(x,\ \alpha x^2+\beta x^{q+1}+r,\ 1)$, $x\in\Fqq$, $r\in\Fq$,   has coordinates 
\begin{equation}
  t_{x,r}=[-2\alpha x+(\beta^q-\beta )x^q,\ 1, \ \alpha x^2-\beta^qx^{q+1}-r].
\Labele{eqn:tgt}
\end{equation}
 The point 
 $\eone$ lies on the tangent $t_{x,r}$ iff
\begin{equation}
\epsilon+\alpha x^2-\beta^qx^{q+1}=r.
\Labele{eqn:f1}
\end{equation}
Raising to the power of $q$ and recalling $\epsilon^q=-\epsilon$  gives 
\begin{equation}
-\epsilon +\alpha^q x^{2q}-\beta x^{q+1}=r^q.
\Labele{eqn:f1b}
\end{equation}
 As $r\in\Fq$, we have $r=r^q$, so  equating  (\ref{eqn:f1}) and (\ref{eqn:f1b}) gives the condition
\begin{equation}
  \alpha x^2-\alpha^qx^{2q}+(\beta-\beta^q)x^{q+1}+2\epsilon=0.\Labele{eqn:f2}
\end{equation}
As $t_{x,r}\cap\Uab=(x,\ \alpha x^2+\beta x^{q+1}+r,\ 1)$, and points in $\pedal(\eone)$  satisfy (\ref{eqn:f1}) and (\ref{eqn:f2}), we conclude that   points  in $\pedal(\eone)$ have   coordinates given by 
\begin{equation}
  Q_x=(x,2\alpha x^2+(\beta-\beta^q)x^{q+1}+\epsilon ,1),\ \ x\in\Fqq \textup{ satisfying (\ref{eqn:f2})}. \Labele{eqn:f3}
\end{equation}
 Points on the line $T_\infty R_1$ have coordinates $E_k=(0,k,1)$ for $k\in\Fqq\cup\{\infty\}$. Thus lines through the point $U_\infty=(1,0,0)$ have form $\ell_k =U_\infty E_k$, so have coordinates  $\ell_k=[0,-1,k]$ for $k\in\Fqq\cup\{\infty\}$. 
 The point    $Q_{x}$ with coordinates given in (\ref{eqn:f3}) lies on the line $\ell_k$ if and only if
$
  (2\alpha x^2+(\beta-\beta^q)x^{q+1}+\epsilon )-k =0
$
  which   by (\ref{eqn:f2}) holds iff $k=\alpha x^2+\alpha^qx^{2q}-\epsilon$. Let $s= \alpha x^2+\alpha^qx^{2q}$, so $k=s-\epsilon$ and $s^q=s$, so $s\in\Fq$. 
    
 Thus each point $Q_x$ with $x\in\Fqq$   satisfying (\ref{eqn:f2}) lies on a line $\ell_k$ with $k=s-\epsilon$ for some $s\in\Fq$. Thus the points of $\pedal(\eone)$ are contained in the lines of the given Baer pencil. 
\end{proof}

\begin{lemma}\Label{18} Let $\ell$ be a line through $U_\infty $. 
 \begin{enumerate}
 \item If $\alpha$ is a  nonsquare in $\Fqq$, then $\ell\cap\pedal(\eone)$ has size $0$ or $2$. 
 \item If $\alpha$ is a nonzero square in $\Fqq$,   then $\ell\cap\pedal(\eone)$ has size $0$, $2$ or $4$.
 Further, there exists  at least  $\frac{q-3}4$  4-secants, and every 4-secant  passes through the point $U_\infty$.  
    \end{enumerate}
\end{lemma}

\begin{proof} By  Lemma~\ref{10},  points of $\pedal(\eone)$ have coordinates $Q_x=(x,2\alpha x^2+(\beta-\beta^q)x^{q+1}+\epsilon ,1)$ where $x\in\Fqq$ satisfies (\ref{eqn:f2}).
Using the basis $\{1,\e\}$ for $\Fqq$ over $\Fq$, 
 we write  $x=x_0+x_1\epsilon$, $\alpha=a_0+a_1\epsilon$, $\beta=b_0+b_1\epsilon$  for unique $x_0,x_1,a_0,a_1,b_0,b_1\in\Fq$. Substituting into (\ref{eqn:f2})  using $\epsilon^q=-\e$ and $\e^2=w \in\Fq$ 
 gives the equation
 $2\epsilon\big((a_1+b_1)x_0^2+w(a_1-b_1)x_1^2+2a_0x_0x_1+1\big)=0$. Dividing by $2\epsilon$ gives an equation with coefficients in $\Fq$. It 
   is the equation of a conic $\mathcal C$ in $\PG(2,q)$
 which has symmetric matrix given by 
 \[
    A_{\mathcal C}=\begin{pmatrix}
    a_1+b_1&a_0&0\\
    a_0&w(a_1-b_1)&0\\
    0&0&1
    \end{pmatrix}.
\]
Recall that $\Uab$ has discriminant $d=(\beta^q-\beta)^2+4\alpha^{q+1}$ a nonsquare in $\Fq$, so in particular $d\neq 0$. Substituting $\alpha=a_0+a_1\epsilon$, $\beta=b_0+b_1\epsilon$ into this expression for $d$ (using $\epsilon^q=-\e$ and $\e^2=w$) gives $d=4a_0^2-4a_1^2w-4b_1^2w\neq0$. Hence
 $\det(A_{\mathcal C})=-d/4$, so $\det(A_{\mathcal C})\neq0$ and 
  $\mathcal C$ is non-degenerate. 
That is, the point $Q_x=(x,2\alpha x^2+(\beta-\beta^q)x^{q+1}+\epsilon ,1)$ lies in $\pedal(\eone)$ iff $x\in\Fqq$ satisfies (\ref{eqn:f2}),   iff
the point $F_x=(x_0,x_1,1)$ in $\PG(2,q)$ lies on the non-degenerate conic $\mathcal C$ of $\PG(2,q)$. 
  
Let $\ell$ be a line through the point $U_\infty $ and suppose $|\ell\cap\pedal(\eone)|>0$; we show that $\ell$ contains exactly $2$ or $4$ points of $\pedal(\eone)$. We can write $\ell=U_\infty E_k$ for some point $E_k=(0,k,1)$, $k\in\Fqq$. 
By the proof of Lemma~\ref{10}, 
the point $Q_x$ in $\pedal(\eone)$ lies on $\ell=U_\infty E_k$ iff $k=s-\epsilon$ and $s=\alpha x^2+\alpha^qx^{2q}$.   
 Substituting   $x=x_0+x_1\epsilon$, $\alpha=a_0+a_1\epsilon$
into $s=\alpha x^2+\alpha^qx^{2q}$ gives the 
equation 
$s=2a_0x_0^2+2wa_0x_1^2 +4wa_1x_0x_1.  
$ As $s\in\Fq$, all coefficients lie in $\Fq$, so this is a conic 
   $\mathcal D_s$  in $\PG(2,q)$ with matrix 
    \[
    A_{ {\mathcal D}_s}=\begin{pmatrix}
    2a_0&2wa_1&0\\
    2wa_1&2wa_0&0\\
    0&0&-s
    \end{pmatrix}.
    \]
That is, the point  $Q_x=(x,2\alpha x^2+(\beta-\beta^q)x^{q+1}+\epsilon ,1)$ in $\pedal(\eone)$ lies on the line  $U_\infty E_{s-\epsilon }$ iff $s=\alpha x^2+\alpha^qx^{2q}$. That is, in $\PG(2,q)$  the point
$F_x =(x_0,x_1,1)$ lies on the  conic $ {\mathcal D}_s$.
Hence the number of points of $\pedal (\eone)$ on the line $\ell=U_\infty E_{s-\epsilon }$ equals the number of points on the intersection of the two conics $\mathcal C$ and $\mathcal D_s$. Thus lines through $U_\infty $ meet $\pedal(\eone)$ in at most four points.

Note that $F_x\in\mathcal C$ iff $F_{-x}\in\mathcal C$; and $F_x\in\mathcal D_s$ iff $F_{-x}\in\mathcal D_s$.
As $F_0 =(0,0,1)$ is not in $\mathcal C$, the intersection $\mathcal C \cap  {\mathcal D}_s$ has an even number of points. That is, every line through $U_\infty $  contains $0$, $2$ or $ 4$ points of $\pedal(\eone)$.

The number of points in the intersection of two conics was  considered by Dickson \cite{dickson}. As $\mathcal C$ is non-degenerate, $\mathcal C$ and $\mathcal D_s$ meet in four distinct points iff the 
 pencil of conics $\mathcal C_\lambda=\mathcal C+\lambda {\mathcal D}_s$ for $\lambda \in \Fq \cup \{\infty\}$  contains three distinct degenerate conics.
 The conic $\mathcal C_\lambda$ has matrix 
     \[
    A_{\mathcal C_\lambda}=A_{\mathcal C}+\lambda A_{{\mathcal D}_s}=\begin{pmatrix}
    (a_1+b_1)+2a_0\lambda&a_0+2wa_1\lambda&0\\
    a_0+2wa_1\lambda&w(a_1-b_1)+2wa_0\lambda&0\\
    0&0&1-\lambda s
    \end{pmatrix}
    \]
    which has determinant
$
  \big(4\lambda^2w(a_0^2-wa_1^2)-(a_0^2-wa_1^2+wb_1^2)\big)(1-\lambda s).
$
As $\mathcal C$ is non-degenerate, this equation is not identically zero, so there are three roots  for $\lambda\in\Fq\cup\{\infty\}$ for which $\det(\mathcal C_\lambda)=0$. One root is $\lambda =1/s$, the other two are roots of 
\begin{eqnarray}\Labele{e5}
4\lambda^2w(a_0^2-wa_1^2)&=&(a_0^2-wa_1^2+wb_1^2).
\end{eqnarray}
The right hand side of (\ref{e5}) is the expansion of $d/4$, which by Result~\ref{Uab} is a nonsquare in $\Fq$. 
The left hand side of (\ref{e5}) is the expansion of $ 4\lambda^2w\alpha^{q+1}$ (using $\alpha=a_0+a_1\epsilon$). 
We determine when this is a nonsquare  in $\Fq$.
Recall that $\zeta$ is a primitive element of $\Fqq$, $\epsilon=\zeta^{(q+1)/2}$, and $w=\epsilon^2$ is a primitive element of $\Fq$. 
Hence $w$ is  a nonsquare in $\Fq$ and $\alpha=\zeta^i$ for some $i\in\{1,\ldots,q^2-1\}$. Hence $\alpha^{q+1}=w^i$.
Thus $\alpha^{q+1}$ is a square in $\Fq$ iff $\alpha$ is a square in $\Fqq$. Hence the left hand side  of (\ref{e5}) is a nonsquare  in $\Fq$ iff $\alpha$ is a square in $\Fqq$. 
Hence (\ref{e5}) has two roots for $\lambda$ in $\Fq$ iff $\alpha$ is a square in $\Fqq$. In this case, the roots are nonzero, so are distinct and are denoted $\pm\lambda_1$. 

Suppose $\alpha$ is a nonsquare in $\Fqq$, then we conclude that the pencil of conics $\mathcal C_\lambda$, $\lambda \in\Fq \cup \{\infty\}$  contains exactly one degenerate conic, namely $\mathcal C_{1/s}$, and so  $\mathcal C$ and ${\mathcal D}_s$ cannot meet in four points. Thus the line $\ell=U_\infty E_{s-\epsilon }$ contains $ 2$ points of $\pedal(\eone)$. We conclude that every line  through $U_\infty $ meets $\pedal(\eone)$ in either $0$ or $2$ points.

Suppose 
$\alpha$ is a square in $\Fqq$, then there are three values (possibly repeated)  for which $\mathcal C_\lambda$ is degenerate, namely $\lambda =1/s,\pm\lambda_1$.  As $\lambda_1\neq0$, there are either two or three distinct values. As $s$ varies in $\Fq$, there are at most two values of $s$ which give two distinct roots, the remaining values of $s$ give three distinct roots. Thus at most two lines of form $U_\infty E_{s-\epsilon }$ meet $\pedal(\eone)$ in two points.   It follows that when $q \equiv 1 \pmod4$, there is one line  through $U_\infty $ meeting $\pedal(\eone)$ in $2$ points and $\frac{q-1}{4}$  lines  through $U_\infty $ meeting $\pedal(\eone)$ in $4$ points;  and when 
$q \equiv 3 \pmod4$, there is either zero or two lines  through $U_\infty $ meeting $\pedal(\eone)$ in $2$ points and $\frac{q+1}{4}$ or $\frac{q-3}{4}$  lines  through $U_\infty $ meeting $\pedal(\eone)$ in $4$ points.
        \end{proof}

\begin{lemma}\Label{17} If $\ell$ is a line of $\PG(2,q^2)$ not through $U_\infty $, then $\ell\cap\pedal(\eone)$ has size $0$, $1$ or $2$. 
\end{lemma}

\begin{proof}   
Let $\ell$ be a line through 
$T_\infty=(0,1,0)$, so $\ell$ has coordinates $[u,0,v]$ for some $u,v\in\Fqq$, not both zero. 
By Lemma~\ref{10}, points in $\pedal(\eone)$   have coordinates of form  $Q_x=(x,2\alpha x^2+(\beta-\beta^q)x^{q+1}+ \epsilon,1)$, with $x\in\Fqq$ satisfying (\ref{eqn:f2}).  The point $Q_x$ lies on $\ell$ if  $ux+v=0$.  That is, $\ell$ contains at most one point of   $\pedal(\eone)$. Moreover, there do exist 1-secants of $\pedal(\eone)$.

Let $\ell$ be a line   not through $U_\infty $ or $T_\infty$, so $\ell$ has  coordinates $[u,1,v]$, $u,v\in\Fqq$, $u,v\neq0$.  The point $Q_x$ lies on $\ell$ if 
$ux+2\alpha x^2+(\beta-\beta^q)x^{q+1}+ \epsilon+v=0$,   that is, by  (\ref{eqn:f2}), $$ux+ax^2+a^qx^{2q}-\epsilon +v=0.$$ 
We expand by writing $u=u_0+u_1\epsilon$, $v=v_0+v_1\epsilon$, $x=x_0+x_1\epsilon$, $\alpha=a_0+a_1\epsilon$, $\beta=b_0+b_1\epsilon$  for unique $u_0,\ldots,b_1\in\Fq$; and equating coefficients of the basis $\{1,\epsilon\}$ to get two equations:
\begin{eqnarray*}
2a_0x_0^2+2a_0wx_1^2+4a_1wx_0x_1+u_0x_0+u_1wx_1+v_0&=&0\\
u_1x_0+u_0x_1+v_1-1&=&0.
\end{eqnarray*}
The first equation is the equation of a conic in $\PG(2,q)$ 
and the second is the equation of a line in $\PG(2,q)$. Hence there are at most two points $F_x=(x_0,x_1,1)$ in their intersection in $\PG(2,q)$. Thus there are at most two values of $x$ for which the point $Q_x$ lies on the line $\ell$, that is, $|\ell\cap\pedal(\eone)|\leq 2$  as required. 
\end{proof}

\begin{lemma} \Label{16a}
\begin{enumerate}
\item If $\alpha$ is a nonsquare in $\Fqq$, then $\pedal(\eone)$ is an arc. 
\item If $\alpha$ is a nonzero square in $\Fqq$,   then $\pedal(\eone)$ is a set of class $(0,1,2,4)$. Further, there exists  at least  $\frac{q-3}4$  4-secants, and every 4-secant  passes through the point $U_\infty$. 
\end{enumerate}
\end{lemma}

\begin{proof} This follows immediately from Lemmas~\ref{18} and \ref{17}. Note that the proof of Lemma~\ref{17} shows that $\pedal(\eone)$ has 1-secants, so in case 2, when $\alpha$ is a  nonzero square in $\Fqq$, each of the intersection numbers $0,1,2,4$ does occur. 
\end{proof}

\subsection{The feet of $\etwo=(0,w \epsilon ,1)$, $q$ odd}

This case is very similar to the one above. Here we will consider the feet of a point $\etwo=(0,w \epsilon ,1)$ where $w= \epsilon^2$.

\begin{lemma}\Label{13} 
\begin{enumerate}
\item The feet of $\etwo=(0,w \epsilon ,1)$ are the points with coordinates given  by $Q_x'=(x,2\alpha x^2+(\beta-\beta^q)x^{q+1}+w \epsilon ,1)$, with $x\in\Fqq$ satisfying $\alpha x^2-\alpha^qx^{2q}+(\beta-\beta^q)x^{q+1}+2w \epsilon =0.$

\item The points of $\pedal(\etwo)$ lie on the lines of the Baer pencil  joining the point vertex $U_\infty =(1,0,0)$ to the Baer subline $\{ E_{s-w\epsilon}=(0,s-w\epsilon,1) ,|\,s \in \Fq\cup\{\infty\}\}$.
\end{enumerate}
\end{lemma}

\begin{proof} The proof is similar to the proof of Lemma~\ref{10}.  The point 
$\etwo$ lies on the tangent $t_{x,r}$ iff 
$
  w \epsilon +\alpha x^2-\beta^qx^{q+1}-r =0.
$
Raising to the power of $q$ and
  using $r=r^q$ gives the condition
\begin{equation}
  \alpha x^2-\alpha^qx^{2q}+(\beta-\beta^q)x^{q+1}+2w \epsilon =0.\Labele{eqn:g2}
\end{equation}
We conclude that points of $\pedal(\etwo)$ are the points with coordinates 
\begin{equation}
Q_x'=(x,2\alpha x^2+(\beta-\beta^q)x^{q+1}+w \epsilon ,1), x\in\Fqq \textup{ satisfying (\ref{eqn:g2}).}\Labele{eqn:g3}
\end{equation}
Recall $E_k=(0,k,1)$, $k\in\Fqq\cup\{\infty\}$ and $\ell_k=U_\infty E_{k }$.
The point    $Q'_{x}$   lies on the line $\ell_k$ if and only if
 $k=\alpha x^2+\alpha^qx^{2q}-w\epsilon=s-w\epsilon$ where  $s= \alpha x^2+\alpha^qx^{2q}\in\Fq$.  
 Thus each point $Q_x'$ with $x\in\Fqq$   satisfying (\ref{eqn:g2}) lies on a line $\ell_k$ with $k=s-w\epsilon$ for some $s\in\Fq$. Thus the points of $\pedal(\etwo)$ are contained in the   lines of the given Baer pencil. 
\end{proof}

 \begin{lemma} \Label{16b}
\begin{enumerate}
\item If $\alpha$ is a nonsquare in $\Fqq$, then $\pedal(\etwo)$ is an arc. 
\item If $\alpha$ is a nonzero square in $\Fqq$,  then $\pedal(\etwo)$ is a set of class $(0,1,2,4)$. Further, there exists  at least  $\frac{q-3}4$  4-secants, and every 4-secant  passes through the point $U_\infty =(1,0,0)$. 
\end{enumerate}
\end{lemma}

\begin{proof} An argument similar to the proof of Lemma~\ref{17} shows that every line not through $U_\infty$ contains $0$, $1$ or $2$ points of $\pedal(\etwo)$. 

Let $\ell$ be a line through $U_\infty$ with $|\ell\cap\pedal(\etwo)|>0$. We use a similar argument to the proof of Lemma~\ref{18}. In this case, the conic $\mathcal C$ has equation  $(a_1+b_1)x_0^2+w(a_1-b_1)x_1^2+2a_0x_0x_1+wx_2^2=0$, the conic $\mathcal D_s$ has equation $2a_0x_0^2+2wa_0x_1^2 -sx_2^2+4wa_1x_0x_1=0$, so
\[
   A_{\mathcal C_\lambda}=A_{\mathcal C}+\lambda A_{{\mathcal D}_s}=\begin{pmatrix}
    (a_1+b_1)+2a_0\lambda&a_0+2wa_1\lambda&0\\
    a_0+2wa_1\lambda&w(a_1-b_1)+2wa_0\lambda&0\\
    0&0&w-\lambda s
    \end{pmatrix}
    \]
and
$\det(\mathcal C_\lambda)=    [2^2\lambda^2w(a_0^2-wa_1^2)-(a_0^2-wa_1^2+wb_1^2)](w-\lambda s)$.
The remaining computations are similar to the proof of Lemma~\ref{18}. 
\end{proof}

    \subsection{Proof of Theorem~\ref{19}}

\emph{Proof of Theorem~\ref{19}.}\ \  Let $\Uab$ be a non-classical BM unital in $\PG(2,q^2)$, $q$ odd.  By Result~\ref{21}, the group $G$ has either one or two orbits on the points of $\PG(2,q^2)\setminus\{\Uab,\li\}$. If there are two orbits, then the points $\eone=(0,\epsilon,1)$, $\etwo=(0,w\epsilon,1)$ lie in different orbits. 
Let $P$ be a point, $P\notin\Uab\cup\li$, so $P$ lies in   the same orbit as either $\eone$ or $\etwo$ (or both if there is only one orbit). 
The Baer pencil property of $\pedal(P)$ follows from Lemmas~\ref{10} and \ref{13}. The structure of 
$\pedal(P)$
   follows from Lemmas~\ref{16a} and \ref{16b}. This completes the proof of Theorem~\ref{19}.
\hfill$\square$

  \subsection{A special case: the feet of a conic-BM unital}

 There is a special class of orthogonal BM unitals in $\PG(2,q^2)$ which can be expressed as a union of conics. These were independently found by Baker and Ebert \cite{bake90} and Hirschfeld and Sz\H{o}nyi \cite{hirs}, called \emph{conic-BM unitals}. Coordinates for conic-BM unitals are determined in \cite{bake92}: the unital $\Uab$ contains a non-degenerate conic iff $\beta\in\Fq$, in which case $\Uab$ is  the union of $q$ non-degenerate conics pairwise meeting in the point $T_\infty=(0,1,0)$.  
 
The next result shows that $\beta\in\Fq$ implies that $\alpha$ is a nonsquare in $\Fqq$. The converse is not true, there are unitals with $\alpha$   a nonsquare in $\Fqq$ that are not conic-BM unitals.

\begin{lemma}\Label{11} Suppose $q$ odd and let $\alpha,\beta\in\Fqq$ satisfy the discriminant condition of Result~\ref{Uab}.  
If $\beta\in\Fq$, then $\alpha$ is a nonsquare in $\Fqq$.
\end{lemma}
\begin{proof}
The discriminant condition is that $d=(\beta^q-\beta)^2+4\alpha^{q+1}$ is a nonsquare in $\Fq$.  If $\beta\in\Fq$ then $\beta=\beta^q$ and the discriminant condition becomes $\alpha^{q+1}$ a nonsquare in $\Fq$. The result follows as  $\alpha^{q+1}$ is a square in $\Fq$ iff $\alpha$ is square in $\Fqq$ (this statement is proved in the proof of Lemma~\ref{18}).
\end{proof}

If $\U$ is a conic-BM unital, then $\U=\Uab$ with $\beta\in\Fq$, so by Lemma~\ref{11},  $\alpha$ is a nonsquare in $\Fqq$. Hence by Theorem~\ref{19}, the pedal set for each point $P\notin\Uab\cup\li$ is an arc. 
Although pedal sets for a conic-BM unital is a corollary of the general case, there is a direct proof in this case that is much shorter. We include this short proof here. The motivation for including the proof is to highlight the difference in this case. It seems possible that there is additional geometrical structure in this case that we have not yet found.

\begin{theorem}
Let  $\U$ be a conic-BM unital wrt $\li$, then  for all $Q\notin\U\cup\li$,  $\pedal(Q)$ is an arc.
\end{theorem}

\begin{proof} Let $\U$ be a conic-BM unital, then by $\cite{bake92}$, $\U$ is projective equivalent to $\Uab$ with $\beta\in\Fq$. Moreover, we can without loss of generality let $\beta=0$.
 That is,   let $\U=\U_{\alpha0}$ for some $\alpha\in\Fqq$ with discriminant $d=4\alpha^{q+1}$ a nonsquare in $\Fq$. 
Let $Q$ be a point not in $\U$ and not on $\li$. Let $A,B,C$ be three points in $\pedal(Q)$ and suppose $A,B,C$ are collinear. We work to a contradiction. Coordinates of points of $\U$ are given in Result~\ref{Uab}.
As the group that fixes $T_\infty$ and $\U$ is transitive on the points of $\U\setminus\{T_\infty\}$, without loss of generality, let $A=P_{0,0}$, and $B=P_{x,r}$, $C=P_{y,s}$ for some $x,y\in\Fqq$, $r,s\in\Fq$. That is, $  A=(0,0,1)$, $  B=(x, \alpha x^2+r,1)$, $  C=(y, \alpha y^2+s,1)$. If $x=y$, then $T_\infty,B,C$ lie on the line with coordinates $[1,0,x]$. By Result~\ref{12}, the points of $\U$ on this line are the pedal set of a point on $\li$, a contradiction as we have supposed the tangents to $\U$ at $B,C$ meet in $Q\notin\li$. So $x\neq y$, and a similar argument shows $x,y\neq0$. 

As $A,B,C$ are collinear, the matrix with rows $  A,  B,   C$ has determinant $0$, and so
$x(\alpha y^2+s)-y(\alpha x^2+r)=0$. We rearrange this to 
\begin{eqnarray}\Labele{e1}
xs-yr&=& y \alpha x^2- x \alpha y^2.
\end{eqnarray}
The equations of the tangents to $\U$ at $A,B,C$ can be computed using (\ref{eqn:tgt})   and have coordinates $[0,1,0]$, $[-2\alpha x, 1,  \alpha x^2-r]$, $[-2\alpha y, 1,  \alpha y^2-s]$ respectively. As these lines are concurrent at $Q$, the matrix with these three coordinates as rows has determinant $0$, and so 
$ 2\alpha x(\alpha y^2-s)- 2\alpha y(\alpha x^2-r)=0$. We rearrange to get $xs-yr=x\alpha y^2-y\alpha x^2$. Equating with (\ref{e1}) gives $2(y \alpha x^2- x \alpha y^2)=0$ and so $2\alpha x y (x-y)=0$. This is a contradiction as we have $q$ odd, $\alpha\neq0$, $x\neq y$ and $x,y\neq 0$. 
We conclude that $A,B,C$ cannot be collinear, and so $\pedal(Q)$ is an arc. 
\end{proof}

\section{The feet of an orthogonal BM unital for $q$ even}

An orthogonal BM unital in $\PG(2,q^2)$  has form $\U_{\alpha\beta}$ defined in Result~\ref{Uab}. To determine the structure of the pedal sets of $\U_{\alpha\beta}$ when $q$ is even, we use the following result.

\begin{result}\Label{21even} \cite{eber92} Let $\Uab$ be an orthogonal BM unital in $\PG(2,q^2)$, $q$ even.  Let $G$ be the subgroup of $\PGammaL(3,q^2)$ that fixes $T_\infty$ and fixes $\U_{\alpha\beta}$. Then $G$ acts transitively on the points of $\U_{\alpha\beta}\setminus\{T_\infty\}$, acts transitively on the points of $\li\setminus\{T_\infty\}$, and has one  orbit on the remaining points of $\PG(2,q^2)$. 
\end{result}

 Let $\Uab$ be a non-classical orthogonal BM unital in $\PG(2,q^2)$, $q$ even.  
We work with the following basis for $\Fqq$. Choose $\{1,\delta\}$ as basis for $\Fqq$ over $\Fq$ where $q\geq 4$, $\delta\in\Fqq\setminus\Fq$, $\delta^q=1+\delta$ and $\delta^2=v+\delta$ where $v\neq1$ and $T(v)=1$ (where   $T(v)=v+v^2+\cdots+v^{2^{h-1}}$ is the \emph{absolute trace} of $v$ in $\Fq$, $q=2^h$). 

We use the point with coordinates $\ee=(0,\delta,1)$ as our representative for a point not on $\Uab$ or on  $\li$. 
 In the next section we determine the structure of $\pedal(\ee)$. 
 
 \subsection{The feet of $\ee=(0,\delta,1)$, $q$ even}

\begin{lemma}\Label{15} 
\begin{enumerate}
\item The feet of $\ee=(0,\delta,1)$ are the points with coordinates given by  $Q_x=(x, \alpha  x^2+\alpha^qx^{2q}+\delta^q,1)$, where $x\in\Fqq$ satisfies $1+\alpha  x^2+\alpha^qx^{2q}+(\beta +\beta ^q)x^{q+1}=0.$
\item   The points of $\pedal(\ee)$ are contained in the lines of the Baer pencil joining the point vertex $U_\infty=(1,0,0)$ to the   Baer subline $\{E_{s+\delta^q}= (0,s+\delta^q,0) \,|\, s \in \Fq\cup\{\infty\} \}$.
\end{enumerate}
\end{lemma} 

\begin{proof} The  tangent to $\Uab$ at the point $  P_{x,r}=(x,\ \alpha x^2+\beta x^{q+1}+r,\ 1)$ has coordinates  $  t_{x,r}=
[(\beta^q+\beta)x^q,\ 1,\ \alpha x^2+\beta^qx^{q+1}+r]$ (see \cite[Lemma 4.27]{book}).  
 The point 
 $\ee=(0,\delta,1)$ lies on the tangent $t_{x,r}$ iff
\begin{equation}
\delta+\alpha  x^2+\beta^qx^{q+1}=r.\Labele{even0}
\end{equation}
Raising to the power of $q$    gives 
\begin{equation}
\delta^q+\alpha  ^qx^{2q}+\beta x^{q+1}=r^q.\Labele{even0b}
\end{equation}
 As $r\in\Fq$, we have $r=r^q$, so  equating (\ref{even0}) and  (\ref{even0b}) gives the condition
\begin{equation}
  1+\alpha  x^2+\alpha^qx^{2q}+(\beta +\beta ^q)x^{q+1}=0.\Labele{even1}
\end{equation}
As $t_{x,r}\cap\Uab=(x,\ \alpha x^2+\beta x^{q+1}+r,\ 1)$, and 
points in $\pedal(\ee)$  satisfy  (\ref{even0b}) and (\ref{even1}), we conclude that the points in $\pedal(\ee)$ have    coordinates given by 
 \begin{equation}
Q_x=(x, \alpha  x^2+\alpha^qx^{2q}+\delta^q,1),\ \ x\in\Fqq \textup{ satisfying  (\ref{even1})}. \Labele{even3}
\end{equation}
 Points on the line $T_\infty \ee$ have coordinates $E_k=(0,k,1)$ for $k\in\Fqq\cup\{\infty\}$. Thus lines through the point $U_\infty=(1,0,0)$ have form $\ell_k =U_\infty E_k$, so have coordinates  $\ell_k=[0,1,k]$ for $k\in\Fqq\cup\{\infty\}$. 
 The point    $Q_{x}$ with coordinates given in (\ref{even3}) lies on the line $U_\infty E_{k }$ if and only if
$
  (\alpha  x^2+\alpha  ^qx^{2q}+\delta^q)+k =0,
$
that is, $k=\alpha x^2+\alpha^qx^{2q}-\delta^q$. Let $s= \alpha x^2+\alpha^qx^{2q}$, so $k=s+\delta^q$ and $s^q=s$, so $s\in\Fq$. 
 Thus each point $Q_x$ with $x\in\Fqq$   satisfying (\ref{even1}) lies on a line $\ell_k$ with $k=s+\delta^q$ for some $s\in\Fq$. Thus the points of $\pedal(\ee)$ are contained in the   lines of the given Baer pencil. 
\end{proof}

\begin{lemma}\Label{22} If $\ell$ is a line through $U_\infty$, then $\ell$ meets $\pedal(\ee)$ in at most two points. 
\end{lemma}

\begin{proof}
By  Lemma~\ref{15},  points of $\pedal(\ee)$ have coordinates $Q_x=(x, \alpha  x^2+\alpha  ^qx^{2q}+\delta^q,1)$ where $x\in\Fqq$ satisfies (\ref{even1}). 
Using the basis $\{1,\delta\}$ for $\Fqq$ over $\Fq$, 
 we write  $x=x_0+x_1\delta$, $\alpha=a_0+a_1\delta$, $\beta=b_0+b_1\delta$  for unique $x_0,x_1,a_0,a_1,b_0,b_1\in\Fq$. 
 Substituting into (\ref{even1})  using  
$\delta^q=1+\delta$ and $\delta^2=v+\delta$,
and writing as a homogeneous equation in $(x_0,x_1,x_2)$ gives the following which is the equation of a conic  in $\PG(2,q)$, denoted $\C$:
\begin{equation}\label{C19} (a_1+b_1)x_0^2+b_1x_0x_1+(a_0+a_1+a_1v+b_1v)x_1^2+x_2^2=0.\end{equation}
As $\beta\notin\Fq$, we have $b_1\neq0$. So the conic $\mathcal C$ has nucleus   $N=(0,0,b_1)$ which does not lie in $\mathcal C$, so $\mathcal C$ is non-degenerate. Thus the point $Q_x$ lies in $\pedal(\ee)$ if in $\PG(2,q)$, the point $F_x=(x_0,x_1,1)$  lies on the non-degenerate conic $\mathcal C$.

Let $\ell$ be a line through the point $U_\infty $ and suppose $|\ell\cap\pedal(\ee)|>0$; we show that $\ell$ contains exactly $2$ or $4$ points of $\pedal(\ee)$. We can write $\ell=U_\infty E_k$ for some point $E_k=(0,k,1)$, $k\in\Fqq$. 
By the proof of  Lemma~\ref{15}, 
the point $Q_x$ in $\pedal(\ee)$ lies on $\ell=U_\infty E_k$ iff $k=s+\delta^q$ and $s=\alpha x^2+\alpha^qx^{2q}$.   
 Substituting   $x=x_0+x_1\epsilon$, $\alpha=a_0+a_1\epsilon$
into $s=\alpha x^2+\alpha^qx^{2q}$ gives the 
equation  $s = a_1x_0^2+(a_0+a_1+a_1v)x_1^2$. This is the equation of a conic $\mathcal D_s$ in $\PG(2,q)$.
Thus the point $Q_x$ lies on the line $\ell$ if the point $F_x=(x_0,x_1,1)$  lies on the   conic $\mathcal D_s$. That is, the number of points in $\pedal(\ee)$ that lie on the line $\ell$ equals the number of points in the intersection $\mathcal C\cap\mathcal D_s$. 
The conic $\mathcal D_s$ has homogeneous equation  $a_1x_0^2+(a_0+a_1+a_1v)x_1^2+sx_2^2=0$. As $q$ is even, this equation  factorises as  $(\sqrt{a_1}x_0+\sqrt{a_0+a_1+a_1v}x_1+\sqrt sx_2)^2$, so $\mathcal D_s$ is a repeated line. Hence $\mathcal D_s$ meets $\mathcal C$ in at most two distinct points  in $\PG(2,q)$.
Hence the number of points of $\pedal(\ee)$ on the line $\ell=U_\infty E_k$ is at most two. 
\end{proof}

\begin{lemma}\Label{23} If $\ell$ is a line not through $U_\infty$, then $\ell$ meets $\pedal(\ee)$ in at most two points. 
\end{lemma}

\begin{proof}
Let $\ell$ be a line through 
$T_\infty=(0,1,0)$, so $\ell$ has coordinates $[u,0,v]$ for some $u,v\in\Fqq$, not both zero. 
By  Lemma~\ref{15}, points in $\pedal(\ee)$   have coordinates of form  $Q_x=(x, \alpha  x^2+\alpha  ^qx^{2q}+\delta^q,1)$, with $x\in\Fqq$ satisfying (\ref{even1}).  The point $Q_x$ lies on $\ell$ if  $ux+v=0$.  That is, the line $\ell$ contains at most one point of $\pedal(\ee)$.   

Let $\ell$ be a line   not through $U_\infty $ or $T_\infty$, so $\ell$ has  coordinates $[u,1,v]$, $u,v\in\Fqq$, $u,v\neq0$.  The point $Q_x$ lies on $\ell$ if 
   \[
   u x+\alpha  x^2+\alpha  ^qx^{2q}+\delta^q+v=0.
   \]
We expand by: writing $u=u_0+u_1\delta$, $v=v_0+v_1\delta$, $x=x_0+x_1\delta$, $\alpha=a_0+a_1\delta$  for unique $u_0,\ldots,a_1\in\Fq$;  equating coefficients of the basis $\{1,\delta\}$;  and writing as homogeneous equations in $(x_0,x_1,x_2)$. This gives us  two equations in $\PG(2,q)$:
   \begin{eqnarray}
 u_0x_0x_2+vu_1x_1x_2+v(a_1x_0^2+(a_0+a_1)x_1^2)+(1+v_0)x_2^2&=&0\\
  u_1 x_0+u_0x_1+(1+v_1) x_2&=&0.\Labele{25}
   \end{eqnarray}
 As $u\neq0$, at least one of $u_0,u_1$ is nonzero, so (\ref{25})  is the equation of a line in $\PG(2,q)$.
This line   meets   the conic $\mathcal C$ whose equation is given in (\ref{C19}) in at most two distinct points  in $\PG(2,q)$. 
Thus there are at most two values of $x$ for which the point $Q_x$ lies on the line $\ell=[u,1,v]$, that is, $|\ell\cap\pedal(\ee)|\leq 2$  as required. 
\end{proof}

\subsection{Proof of Theorem~\ref{24}}

\emph{Proof of Theorem~\ref{24}.}\ \ Let $\Uab$ be a non-classical BM unital in $\PG(2,q^2)$, $q$ even.  By Result~\ref{21even}, the group $G$ has   one  orbit  on the points of $\PG(2,q^2)\setminus\{\Uab,\li\}$.  
Let $P$ be a point, $P\notin\Uab\cup\li$, so $P$ lies in   the same orbit as   $\ee=(0,\delta,1)$. 
The Baer pencil property of $\pedal(P)$ follows from Lemma~\ref{15}; and $\pedal(P)$ is an arc  by   Lemmas~\ref{22} and \ref{23}. This completes the proof of Theorem~\ref{24}.
\hfill$\square$

\section{Comments}

We used Magma \cite{magma} to search for additional geometric properties of a pedal set for small values of $q$, and briefly discuss here.  

We first consider whether a pedal set is contained in a Baer subplane. Let $\U$ be a non-classical orthogonal BM unital wrt $\li$ in $\PG(2,q^2)$ and let $Q\notin\U$, $Q\notin\li$. Magma searches show that in $\PG(2,q^2)$ with  $4\leq q\leq 32$,  $\pedal(Q)$ meets a Baer subplane in at most four points. This leads us to pose the following two conjectures.

\emph{Conjectures} 
\begin{enumerate}
\item  Let $\U$ be an ovoidal BM unital wrt $\li$ in $\PG(2,q^2)$.  Let $Q$ be a point $Q\notin\U$, $Q\notin\li$ and $\B$ a Baer subplane. If $\U$ is nonclassical, then $|\pedal(Q)\cap\B|\leq 4$. 
\item A unital $\U$ in $\PG(2,q^2)$ is classical iff for  every point $Q\notin\U$, $\pedal(Q)$ is contained in a Baer subplane. 
\end{enumerate}


Next we consider the geometrical structure of a pedal set in the Bruck-Bose representation of $\PG(2,q^2)$ in $\PG(4,q)$. Let $\U$ be a non-classical BM unital wrt $\li$ in $\PG(2,q^2)$ and let $Q\notin\U\cup\li$. 
Then  the points of  $\pedal(Q)$ lie on the lines of a Baer pencil with vertex a point $U_\infty\in\li$ and base a Baer subline that contains the point $T_\infty=(0,1,0)\in\li$. 
(We  showed this for the orthogonal case in Lemmas~\ref{10}, \ref{13}, \ref{15}.   A similar  argument can be used  to prove this in the case  when $\U$ is the Buekenhout-Tits unital.) 
In the Bruck-Bose representation of $\PG(2,q^2)$ in $\PG(4,q)$, the set of points on such a Baer pencil corresponds to the set of points in a $3$-space. Let $\K_Q$ denote the set of $q+1$ points in $\PG(4,q)$ corresponding to the points of $\pedal(Q)$. So $\K_Q$ is a set of $q+1$ points lying in a $3$-space and in an orthogonal cone. That is, $\K_Q$ is contained in a (3-dimensional) conic cone. An interesting question is whether $\K_Q$ satisfies any additional geometrical conditions in $\PG(4,q)$. We used Magma \cite{magma} to look at the pedal set for small values of $q$. Magma results show that if  $4\leq q\leq 32$, then for each point $Q$, every plane of $\PG(4,q)$ meets $\K_Q$ in either $0$, $1$, $2$, $3$ or $4$ points, moreover the intersection number $4$ does occur. 
 
%
%
%
%

\bigskip\bigskip

{\bfseries Author information}

S.G. Barwick. School of Mathematical Sciences, University of Adelaide, Adelaide, 5005, Australia.
susan.barwick@adelaide.edu.au

W.-A. Jackson. School of Mathematical Sciences, University of Adelaide, Adelaide, 5005, Australia.
wen.jackson@adelaide.edu.au

P. Wild. Royal Holloway, University of London, TW20 0EX, UK. peterrwild@gmail.com

\end{document}